\documentclass[10pt]{amsart}
\usepackage{amsmath, amssymb, amscd, mathtools, graphicx, amsfonts, enumerate, mathrsfs, xcolor}

\newtheorem{theorem}{Theorem}[section]

\newtheorem{lemma}[theorem]{Lemma}
\newtheorem{corollary}[theorem]{Corollary}

\theoremstyle{definition}

\newtheorem{remark}[theorem]{Remark}

\numberwithin{equation}{section}

\newcommand{\N}{\mathbb{N}}                        
\newcommand{\K}{\mathbb{K}}                        
\newcommand{\R}{\mathbb{R}}                        
\newcommand{\C}{\mathbb{C}}                        

\newcommand{\supp}{\mathrm{supp}}                  
\newcommand{\inexp}{\mathrm{exp}}                  
\newcommand{\NN}{\mathfrak{N}}                     

\newcommand{\mm}{\mathfrak{m}}                     

\newcommand{\verts}{\mathcal{V}}
\newcommand{\Fcal}{\mathcal{F}}

\newcommand{\projd}{\mathrm{pd}}
\newcommand{\Kr}{\mathbb{K}[[x]]}
\newcommand{\Kra}{\mathbb{K} \langle x \rangle}

\begin{document}

\title{Algebraic Approximation of Cohen-Macaulay Algebras}

\author{Aftab Patel}
\address{Department of Mathematics, The University of Western Ontario, London, Ontario, Canada N6A 5B7}
\email{apate378@uwo.ca}


\begin{abstract}	
	This paper shows that Cohen-Macaulay algebras can be algebraically approximated 
	in such a way that their Cohen-Macaulayness and minimal Betti numbers are preserved. This is 
	achieved by showing that finitely generated modules over power series rings can be 
	algebraically approximated in a manner that preserves their diagrams of initial exponents and their 
	minimal Betti numbers. These results are also applied to obtain an 
	approximation result for flat homomorphisms from  
	rings of power series to Cohen-Macaulay algebras. 
\end{abstract}
\maketitle


\section{Introduction}

Throughout what follows $\K$ will denote an arbitrary field, unless specified
otherwise. Also for a fixed integer $n$, $x$ will denote the $n$-tuple 
of variables $(x_1, \dots, x_n)$. The ring of formal power series in variables $x$ with 
coefficients in $\K$ will be denoted by $\Kr$. A formal power 
series $F \in \Kr$ is called \emph{algebraic} if it satisfies 
a non-trivial polynomial relation. The set of all such power series 
forms a ring that is called the \emph{ring of algebraic power series} and is 
denoted by $\Kra$. 
In the case when $\K$ is a complete 
real valued field, the notation $\K\{x\}$ is used 
for the ring of convergent power series in variables $x$. 
Further, throughout this paper, for a $\Kr$-module $M$, the notation $\dim M$ 
will be used for Krull dimension of $M$. 

Let $M \subseteq \Kr^p$ be a module generated by $F_1, \dots, F_s \in \Kr^p$. A 
module $M_{\mu}$ is called an \emph{algebraic approximation of $M$ of order $\mu$} if 
$M_{\mu}$ is generated by power series vectors whose entries are algebraic power series, 
and that agree with the generators 
$F_1, \dots, F_s$ up to order $\mu$. The main result proved in this paper,
Theorem \ref{thm:main}, establishes 
the existence of algebraic approximations of 
arbitrarily high order to modules $M \subseteq \Kr^p$, that
share specified algebraic properties with $M$. These algebraic properties 
are Hironaka's Diagram of Initial exponents
(see Section \ref{sec:diag} for a precise definition), and the 
minimal Betti numbers of $M$. 
(Recall that the minimal Betti numbers are 
the ranks of the free modules appearing in a minimal free resolution of $M$.) 
Aside from standard facts in commutative algebra, the two main tools used to 
prove this theorem are Artin's Approximation Theorem \cite[Theorem 1.10]{Ar2}, 
and the theory of standard bases of modules $M\subseteq \Kr$, which is 
analogous to the theory of Gr\"{o}bner bases of modules over polynomial 
rings and was developed by T. Becker in \cite{Be1}. 

The motivation for the choice of the particular algebraic properties 
above is 
the generalization to arbitrary fields of the result \cite[Theorem 8.1]{JP}, proved 
by J. Adamus and the author, on the existence 
of arbitrarily high order algebraic approximations
of Cohen-Macaulay local analytic algebras over $\R$ or $\C$ that preserve the Hilbert-Samuel 
function. The diagram of initial exponents of an ideal completely  
determines its Hilbert-Samuel function (see Theorem \ref{lem:hsfunc}).
Recall that for an ideal $I \subseteq \Kr$ the \emph{Hilbert-Samuel} function of it is defined as 
\begin{equation*}
	H_I(\eta) = \dim \Kr/I + \mm^{\eta + 1}, \text{ for } \eta \in \N,
\end{equation*}
where $\mm$ is the maximal ideal of $\Kr$. 
Also, recall that for $\K = \R$ or $\C$ a local analytic algebra is a 
quotient of the form $\K\{x\}/I$ where $I \subseteq \K\{x\}$ is an 
ideal. The property of Cohen-Macaulayness is related to the 
minimal Betti numbers of $\K\{x\}/I$; specifically, it can 
be determined by the number of non-zero Betti numbers (see Theorem \ref{thm:cmthm}). 
The result proved in this paper, that achieves the 
generalization of \cite[Theorem 8.1]{JP}, is Theorem \ref{thm:cmbettiapprox}, 
and follows directly from Theorem \ref{thm:main}. In fact, Theorem \ref{thm:cmbettiapprox} 
achieves more than just the generalization of \cite[Theorem 8.1]{JP} to 
arbitrarily fields, as the minimal Betti numbers are 
a finer grained property than Cohen-Macaulayness. As a consequence of 
this, an approximation result analogous to \cite[Theorem 8.1]{JP} for 
Gorenstein local algebras is obtained as a corollary to Theorem \ref{thm:cmbettiapprox}. 
This is Corollary \ref{cor:gorensteinapprox}. It should be noted here that 
generalization of the finite determinacy result proved by J. Adamus and the author in \cite{JP} for 
algebras $\K\{x\}/I$ that are complete intersections, \cite[Theorem 7.3]{JP}, to 
arbitrary fields, already exists and was proved by Srinivas and Trivedi in \cite{ST} using 
techniques different from those used in \cite{JP}. In fact, an 
even stronger result exists in the special case of 
isolated complete intersection singularities, due to Greuel and Pham \cite[Theorem 1.5]{Gre}. 

A generalization of the algebraic counterpart of the result on flat maps from \cite[Theorem 1.2]{AYP} is 
also obtained as a by-product of the proofs of Theorem \ref{thm:main} and Theorem \ref{thm:cmbettiapprox}
(Theorem \ref{thm:flatapprox}). Briefly, \cite[Theorem 1.2]{AYP} states that given a flat analytic map from 
a real or complex analytic germ whose local ring 
is Cohen-Macaulay into a germ of euclidean space, one can find arbitrarily high order 
algebraic approximations to the domain that are Cohen-Macaulay, and to the map that 
are flat, which preserve the Hilbert-Samuel function of the special fibre. In 
this paper a homomorphism of rings $\phi:A \rightarrow B$ will be called 
\emph{flat} if it makes $B$ into a flat $A$-module, and in the case 
when $A$ and $B$ are local rings with maximal ideals $\mm_A$ and $\mm_B$ 
respectively, the ring $B/\phi(\mm_A)B$ will be called the 
\emph{special fibre} of $\phi$. With these definitions 
Theorem \ref{thm:flatapprox} can be stated as follows:
Let $y$ denote 
the $m$-tuple of variables $(y_1, \dots, y_m)$. If one has a flat homomorphism from $\K[[y]]$ to 
$\K[[x]]/I$, and if $\K[[x]]/I$ is Cohen-Macaulay, 
then there exist arbitrarily high order
algebraic approximations to the homomorphism and $\K[[x]]/I$ such that the Hilbert-Samuel 
function and minimal Betti numbers of the special fibre of the approximants is 
the same as those of the special fibre of the original homomorphism and 
such that the approximating homomorphisms are flat. The precise notion of 
algebraic approximation in the context of homomorphisms of rings is defined 
in a manner analogous to that for modules (see the statement of 
Theorem \ref{thm:flatapprox}). The result \cite[Theorem 1.2]{AYP} is a 
direct consequence of Theorem \ref{thm:flatapprox}, and remarkably, 
its proof via Theorem \ref{thm:main} and Theorem \ref{thm:cmbettiapprox}
has fewer dependencies and is conceptually simpler than that of 
\cite[Theorem 1.2]{AYP}. 


The structure of this paper is as follows: Section \ref{sec:bkg} presents relevant definitions 
and theorems used in the proofs of the main theorems of the paper. Section \ref{sec:main} presents
the proof of Theorem \ref{thm:main}. Section \ref{sec:conv} provides 
justification for the validity of the results of 
this paper for rings of convergent power series in the case 
when the field $\K$ is a complete real valued field. Section \ref{sec:appl} presents the applications of 
Theorem \ref{thm:main}, specifically the proof of Theorem \ref{thm:cmbettiapprox} in 
Subsection \ref{sec:cmapprox}, and the proof of Theorem \ref{thm:flatapprox} in 
Subsection \ref{sec:flatapprox}.

\subsection*{Acknowledgements}
The author would like to thank Toshizumi Fukui whose comments helped in the 
conception of
the main idea of the proof of Theorem \ref{thm:main}. 
The author would also like to thank Gert-Martin Greuel, for pointing out the 
result referenced in the paper \cite{Gre} and also, suggesting that the 
author attempt to generalize the results in \cite{AYP} to the case of arbitrary fields. 
This work was done while the author was employed as a Research Associate at the University
of Western Ontario under the supervision of Janusz Adamus. The author's salary was paid 
out of funds associated with NSERC (Natural Sciences and Engineering Research Council of Canada) grant no. RGPIN-2018-04239 awarded to Janusz Adamus. 

\section{Background}
\label{sec:bkg}
\subsection{Power series vectors} 
For $\alpha = (\alpha_1, \dots, \alpha_n) \in \N^n$, 
the monomial $x_{1}^{\alpha_1} \cdots x_n^{\alpha_n}$ will be denoted by $x^{\alpha}$. 
For an integer $p \in \N$, let $e_1, \dots, e_p \in \K[[x]]^p$ denote column vectors with 
$1$ in the $i$th entry and zeros everywhere else, these are called the 
\emph{standard basis vectors of $\Kr^p$}. An element of 
$\K[[x]]^p$ of the form $x^{\alpha}e_i$ for $(\alpha, i) \in \N^{n} \times \{1, \dots, p\}$ is called a \emph{monomial term} and 
$(\alpha, i)$ is called its \emph{exponent}. 
Let $F \in \K[[x]]^p$ be a power series vector. 
With the above notation $F$ can be expressed as a sum of monomial terms as follows 
\begin{equation*}
	F = \sum_{(\alpha, i) \in \N^n \times \{1, \dots, p\}} f_{\alpha,i} x^{\alpha} e_i.
\end{equation*}
The \emph{support of $F$} is $\supp (F) = \{ (\alpha, i) \in \N^n \times \{1, \dots, p\} : f_{\alpha, i} \neq 0\}$. 
If $\mu \in \N$ then 
\begin{equation*}
	j^{\mu} F = \sum_{(\alpha, i) \in N^n \times \{1, \dots, p\},\\|\alpha| \leq \mu} f_{\alpha,i} x^{\alpha} e_i.
\end{equation*}
where $|\alpha| = \alpha_1 + \dots + \alpha_n$, is called the \emph{$\mu$-jet} of $F$. 

\subsection{The diagram of initial exponents}
\label{sec:diag}

For $\alpha = (\alpha_1, \dots, \alpha_n) \in \N^n$, let $|\alpha| = \sum_{i = 1}^n \alpha_i$. 
Then, the lexicographic ordering of the $n+2$-tuples $(|\alpha|, j, \alpha_1, \dots, \alpha_n)$,
where $1 \leq j \leq p$, defines a total ordering on $\N^{n} \times \{1, \dots, p\}$.
The \emph{initial exponent} of $F \in \Kr^p$ is $\inexp(F) = \min \{ (\alpha, i): (\alpha, i) \in \supp(F) \}$, 
where the minimum is taken with respect to the total ordering just defined. If 
$M \subseteq \Kr^p$ is a finitely generated module then the 
\emph{diagram of initial exponents} of $M$ is $\NN(M) = \{\inexp(F): F \in M\}$. 
There exists a unique smallest (finite) set $\verts(M) \subseteq \NN(M)$ such that 
$\NN(M) = \verts(M) + \N^n = \{(\alpha + \beta, i) : (\alpha, i) \in \verts(M), \beta \in \N^n\}$ (\cite[Lemma 3.8]{BM1}). 
The elements of $\verts(M)$ are called the \emph{vertices} of the diagram $\NN(M)$.

In the case when the module under consideration is an ideal $I \subseteq \Kr$, the 
diagram of initial exponents is related to the Hilbert-Samuel function of the 
ideal $I$ as follows: 
\begin{lemma}[{\cite[Lemma 6.2]{JP}}]
	\label{lem:hsfunc}
	Let $I \in \K[[x]]$ be an ideal. Then, 
	\begin{equation*}
		H_I(\eta)=\#\{\beta\in\N^n\setminus\NN(I):|\beta|\leq\eta\},\quad\mathrm{for\ all\ }\eta\geq1\,.
	\end{equation*}
\end{lemma}
\begin{remark}
	\label{rem:HSpoly}
	\begin{itemize}
		\item[(i)] For large values of $\eta \in \N$, the Hilbert-Samuel function $H_I(\eta)$ coincides with a 
			polynomial called the \emph{Hilbert-Samuel polynomial}.  
		\item[(ii)] The degree of the Hilbert-Samuel polynomial of an ideal $I\subseteq \Kr$ is equal to $\dim \Kr/I$. 
	\end{itemize}
\end{remark}

\subsection{Standard bases and standard representations} 
\label{sec:sb}
In what follows a module in $\Kr^p$ generated by $F_1, \dots, F_s \in \Kr^p$, 
will be denoted by $(F_1, \dots, F_s)$. 
Let $M = (F_1, \dots, F_s) \subseteq \Kr^p$ be a finitely generated module. 
A set $\{ G_1, \dots, G_t \} \subseteq M$ is called a \emph{standard basis for $M$}
if $\verts(M) \subseteq \{ \inexp(G_1), \dots, \inexp(G_t) \}$. It is a 
consequence of Hironaka's Division Theorem \cite[Theorems 3.1, 3.4]{BM1} and \cite{Be3} that 
a standard basis for $M$ generates it. 

If $F, H_1, \dots, H_r \in \Kr^p$, then $F$ has a \emph{standard representation} in 
terms of $H_1, \dots, H_r$ if there exist $Q_1, \dots, Q_r \in \Kr$ such 
that, 
\begin{equation*}
	F = \sum_{i = 1}^{r} Q_i H_i \text{ and } \inexp(F) = \min\{\inexp(Q_i H_i): i = 1, \dots, r\}.
\end{equation*}
In the above, by convention, it is assumed that $\inexp(F) < \inexp(0)$ for all $F \neq 0$. 

Suppose now that $F = \sum_{(\alpha, i) \in \N^n \times \{1, \dots, p\}} f_{\alpha, i} x^{\alpha} e_i$, and 
that \\ $G = \sum_{(\alpha, i) \in \N^n \times \{1, \dots, p\}} g_{\alpha, i} x^{\alpha} e_i$. 
Let $(\alpha_F, i_F) = \inexp(F)$, $(\alpha_G, i_G) = \inexp(G)$,  
and $x^{\gamma} = \mathrm{lcm}(x^{\alpha_F}, x^{\alpha_G})$. Define, 
\begin{equation*}
	P_{F, G} = 
	\begin{cases}
		f_{\alpha_F, i_F} x^{\gamma - \alpha_G} \text{ if } i_F = i_G, \\
		0, \text{ otherwise, }
	\end{cases}
\end{equation*}
\begin{equation*}
	P_{G, F} = 
	\begin{cases}
		g_{\alpha_G, i_G} x^{\gamma - \alpha_F} \text{ if } i_F = i_G, \\
		0, \text{ otherwise }
	\end{cases}
\end{equation*}
With the above, the \emph{s-series vector} of $F$ and $G$ is $S(F, G) = P_{F, G} F - P_{G, F} G$. 
The following theorem, 
which follows directly from the corresponding result for 
ideals in power series rings \cite[Theorem 4.1]{Be1}, gives a criterion for determining when a collection of 
power series vectors in $\Kr^p$ forms a standard basis for the module generated 
by it. 
\begin{theorem}[{cf. \cite[Theorem 4.1]{Be1}}]
	\label{thm:becker}
	If $S$ is a finite subset of $\Kr^p$ then the elements of $S$ form a 
	standard basis for the module generated by them if and only if 
	for each pair $G_1, G_2 \in S$, the s-series vector $S(G_1, G_2)$ has a 
	standard representation in terms of elements of $S$. 
\end{theorem}

\begin{remark}
	\label{rem:sbb}
	\begin{itemize}
		\item[(i)] This is a power series analogue of the corresponding 
			result for submodules of free modules over $\K[x]$ called 
			Buchberger's criterion \cite{Eis}. 
		\item[(ii)] The result \cite[Theorem 4.1]{Be1}, is for the case 
			of ideals generated by collections of power series 
			in $\Kr$, however, the proof of the corresponding result, Theorem
			\ref{thm:becker} above, for modules follows by arguments that are 
			almost identical to those used in the proof of \cite[Theorem 4.1]{Be1}. 
		\item[(iii)] The theory of standard basis and the diagram of initial exponents
			in Sections \ref{sec:sb} and \ref{sec:diag} can be developed for 
			orderings other than the one used here, however, this will not 
			be used in this paper.  
		\item[(iv)] Suppose that $F_1, \dots, F_s$ is a standard basis for 
			 a module $M \subseteq \Kr^p$ and that $G_1, \dots, G_s$ is a standard basis for
			 a module $N \subseteq \Kr^p$, then $\{\inexp (G_1), \dots, \inexp (G_s) \} = 
			 \{\inexp (F_1), \dots, \inexp (F_s)\}$ implies that $\NN(M) = \NN(N)$. This
			 follows directly from Hironaka's Division Theorem \cite[Theorem 3.1]{BM1}, \cite{Be3}.
	\end{itemize}
\end{remark}

\subsection{Facts on homomorphisms between free modules over $\Kr$}
By choosing a suitable basis of $\Kr^m, \Kr^n$ for $n, m \in \N$, a homomorphism $\phi: \Kr^m \rightarrow \Kr^n$ 
can be represented by a matrix of dimension $n \times m$ with entries in $\Kr$. 
A homomorphism $\phi:\Kr^m \rightarrow \Kr^n$ is called \emph{algebraic} if the matrix of $\phi$ has 
entries in $\Kra$. The 
following condition for the injectivity of $\phi$ follows immediately from \cite[Proposition I.2.9]{BB},
\begin{lemma}
	\label{lem:injective}
	A homomorphism $\phi: \Kr^m \rightarrow \Kr^n$ is injective if and only if none of the
	the $m\times m$ 
	minors of a matrix representation of $\phi$ is 0.
\end{lemma}
If $\phi:\Kr^m \rightarrow \Kr^n$ is given by the following matrix, 
\begin{equation*}
	\phi = 
	\begin{pmatrix}
		| & & | \\
		S_1 & \cdots &  S_m \\
		| & & |
	\end{pmatrix}
\end{equation*}
where $S_1, \dots, S_m \in \Kr^{n}$ then the \emph{$\mu$-jet} of 
$\phi$ is, 
\begin{equation*}
	j^{\mu}\phi = 
	\begin{pmatrix}
		| & & | \\
		j^{\mu}S_1 & \cdots &  j^{\mu}S_m \\
		| & & |.
	\end{pmatrix}
\end{equation*}

\section{Free resolutions} 
\label{sec:freeres}
Throughout this section let $M \subseteq \Kr^p$ be a finitely generated module. 
For some integer $c$, and integers $n_0, \dots, n_c$ an exact sequence of the form 
\begin{equation}
	\label{eq:freeresex}
	\Fcal_{M}: 0 \xrightarrow{} \Kr^{n_c} \xrightarrow{\phi_{c}} \Kr^{n_{c-1}} \xrightarrow{\phi_{c-1}} \cdots
	\xrightarrow{\phi_{1}} \Kr^{n_0} \xrightarrow{\phi_0} M \xrightarrow{} 0
\end{equation}
is called a \emph{finite free resolution of $M$}. By \cite[Corollary 19.6]{Eis}, all finitely generated 
$\Kr$-modules have a finite free resolution. The number $c$ in the above is called the \emph{length} 
of the free resolution $\Fcal_{M}$. 
If the homomorphisms $\phi_i$ for $ 1 \leq i \leq c$ in 
a finite free resolution $\Fcal_{M}$ have the property that $\mathrm{im}(\phi_{i}) \subseteq \mm \Kr^{n_{i-1}}$, 
then $\Fcal_{M}$ is called a \emph{minimal free resolution} of $M$. By \cite[Lemma 19.4]{Eis} the minimality 
of $\Fcal_{M}$ is equivalent to the condition that for each $1 \leq i \leq c$, a basis for $\Kr^{n_{i-1}}$ 
maps onto a minimal set of generators (i.e., a set of generators with 
minimal cardinality) of $\mathrm{coker}(\phi_i)$. Further, by \cite[Theorem 20.2]{Eis}, 
all minimal free resolutions of $M$ have the same length, which is called the \emph{projective dimension of $M$}, 
and denoted by $\projd_{\Kr}(M)$. The numbers $n_0, n_1, \dots, n_c$ in the minimal free resolution 
$\Fcal_{M}$ of a finitely generated module $M \subseteq \Kr^p$ are called the \emph{minimal Betti numbers} 
of $M$ and are denoted by $\beta^M_0, \dots, \beta^M_c$. 

\begin{remark}
	\label{rem:idealres}
Suppose that $I \subseteq \Kr$ is an ideal. Then there exists an exact sequence $0 \rightarrow I \xhookrightarrow{\iota} \Kr \xrightarrow{\pi} \Kr/I \rightarrow 0$, where $\iota$ is the canonical inclusion of $I$ in $\Kr$ and $\pi$ is the canonical homomorphism from $\Kr$ to $\mathrm{coker} (\iota) = \Kr/I$. Consequently, if 
\begin{equation*}
	\Fcal_{I}: 0 \xrightarrow{} \Kr^{n_c} \xrightarrow{\phi_{c}} \Kr^{n_{c-1}} \xrightarrow{\phi_{c-1}} \cdots
	\xrightarrow{\phi_{1}} \Kr^{n_0} \xrightarrow{\phi_0} I \xrightarrow{} 0
\end{equation*}
is a minimal free resolution of $I$, then 
\begin{equation*}
	\Fcal_{\Kr/I}: 0 \xrightarrow{} \Kr^{n_c} \xrightarrow{\phi_{c}} \cdots
	\xrightarrow{\phi_{1}} \Kr^{n_0} \xrightarrow{\phi_0} \Kr \xrightarrow{\pi} \Kr/I \rightarrow 0
\end{equation*}
is a minimal free resolution of $\Kr/I$. Note that in the above, by abuse of notation, 
$\phi_0$ is used to denote both the homomorphism $\phi_0:\Kr^{n_0} \rightarrow I$ and 
the composition $\iota \circ \phi_0 : \Kr^{n_0} \rightarrow \Kr$. 
\end{remark}

The following characterization of Cohen-Macaulay rings of the form $\Kr/I$ 
follows directly from \cite[Corollary 19.15]{Eis} and the fact that 
$\Kr$ is a regular local ring. 
\begin{theorem}[{cf. \cite[Corollary 19.15]{Eis}}] 
	\label{thm:cmthm}
	If $I \subseteq \Kr$ is an ideal, then $\Kr/I$ is Cohen-Macaulay if and only if 
	\begin{equation*}
		\projd_{\Kr}(\Kr/I) = \dim \Kr - \dim \Kr/I.
	\end{equation*}
\end{theorem}
If $\Kr/I$ is Cohen-Macaulay, and has a minimal free resolution, such as $\Fcal_{\Kr/I}$ above, 
then the last minimal Betti number $\beta^{\Kr/I}_{c+1} = n_c$ of $\Kr/I$ is called the 
\emph{Cohen-Macaulay type} of $\Kr/I$. 

Given $F_1, \dots, F_s \in \K[[x]]^p$, the module generated by all $H = (H_1, \dots, H_s)^{T}\in \K[[x]]^s$ such that
$\sum_{i = 1}^{s}H_i F_i = 0$ is called the \emph{module of syzygies on $F_1, \dots, F_s$} and is 
denoted by $Syz(F_1, \dots, F_s)$. The specification of a minimal free resolution such as \eqref{eq:freeresex} for a 
module $M \subseteq \Kr^p$ is equivalent to the specification of certain syzygy modules as follows: 
If the homomorphism $\phi_0$ is given by, 
\begin{equation*}
	\phi_0 = 
	\begin{pmatrix}
		| & & | \\
		S_{1, 0} & \cdots & S_{n_0, 0}\\
		| & & | 
	\end{pmatrix},
\end{equation*}
then $S_{1, 0}, \dots, S_{n_0, 0} \in \Kr^p$ are a minimal basis of generators of the module $M$. 
For $1 \leq k \leq c$, if the homomorphism $\phi_{k}$ is given by the matrix, 
\begin{equation*}
	\phi_k = 
	\begin{pmatrix}
		| & & | \\
		S_{1, k} & \cdots & S_{n_k, k}\\
		| & & | 
	\end{pmatrix},
\end{equation*}
then $S_{1, k}, \dots, S_{n_k, k} \in \Kr^{n_{k-1}}$ are a minimal basis of generators of the
module $Syz(S_{1, k-1}, \dots, S_{n_{k-1},k-1})$. 
If $M = (F_1, \dots, F_s)$ and $F_1, \dots, F_s$ form a standard basis,  the following theorem gives 
us a basis of generators for $Syz(F_1, \dots, F_s)$. 

\begin{theorem}[{cf. \cite[Theorem 15.10]{Eis}, \cite[Theorem 6.2]{BM1}}]
	\label{thm:schreyer}
	Suppose that $F_1, \dots, F_s \in \K[[x]]^p$ form a standard basis of the module they generate. 
	Further, suppose that the s-series vectors of pairs $F_i, F_j$ for 
	$1\leq i < j \leq s$ are $S(F_i, F_j) = P_{i, j} F_i - P_{j, i} F_j$ and that $Q_m^{i,j} \in \K[[x]]$, for $1\leq m \leq s$ and $1\leq i < j \leq s$, 
	are such that the following are 
	standard representations of the s-series vectors of pairs in $F_1, \dots, F_s$: 
	\begin{equation*}
		P_{i, j}F_i - F_{j, i}F_j = \sum_{m = 1}^{s} Q_{m}^{i,j} F_m \text{ for } 1 \leq i < j \leq s
	\end{equation*}
	Then, 
	\begin{equation*}
		\Xi_{i, j} = P_{i, j} e_i - P_{j, i} e_j - \sum_{m = 1}^{s} Q_{m}^{i, j} e_m \text{ for } 1 \leq i < j \leq s
	\end{equation*}
	form a basis of generators of the module $Syz(F_1, \dots, F_s)$, where the $e_i$ for $1 \leq i \leq s$ are 
	the standard basis vectors of $\Kr^s$. 
\end{theorem}

\begin{remark}
	\begin{itemize}
		\item[(i)] In fact, \cite[Theorem 15.10]{Eis} and \cite[Theorem 6.2]{BM1} also provide an ordering 
			with respect to which the $\Xi_{i,j}$ form a standard basis for $Syz(F_1, \dots F_s)$, 
			however this fact is not used in this paper.  
	\end{itemize}
\end{remark}

\section{Approximation of Modules}
\label{sec:main}

\begin{theorem}
	\label{thm:main}
	Let $M = (F_1, \dots, F_s) \subseteq \Kr^p$ be a finitely generated module. 
	Further, let 
	\begin{equation*}
		\Fcal_{M}: 0 \xrightarrow{} \Kr^{n_c} \xrightarrow{\phi_{c}} \Kr^{n_{c-1}} \xrightarrow{\phi_{c-1}} \cdots
	\xrightarrow{\phi_{1}} \Kr^{n_0} \xrightarrow{\phi_0} M \xrightarrow{} 0
	\end{equation*}
	be a minimal free resolution of $M$. Then there is an integer $\mu_0 \in \N$ such that for each 
	$\mu \geq \mu_0$ there exist $F^{(\mu)}_1, \dots, F^{(\mu)}_s \in \Kra^p$, which 
	generate a module $M_{\mu} \subseteq \Kr^p$, and algebraic homomorphisms 
	$\phi^{(\mu)}_j : \Kra^{n_j} \rightarrow \Kra^{n_{j-1}}$ for $1 \leq j \leq c$, $\phi_0^{(\mu)}:\Kra^{n_0} 
	\rightarrow M_{\mu}$ such that,
	\begin{itemize}
		\item[(i)] $j^{\mu} F_i = j^{\mu}F^{(\mu)}_i$ for $1 \leq i \leq s$. 
		\item[(ii)] $j^{\mu} \phi_i = j^{\mu} \phi^{(\mu)}_i$ for $0 \leq i \leq c$. 
		\item[(iii)] $\NN(M) = \NN(M_{\mu})$. 
		\item[(iv)] The following is a minimal free resolution of $M_{\mu}$
			\begin{equation*}
		\Fcal_{M_{\mu}}: 0 \xrightarrow{} \Kr^{n_c} \xrightarrow{\phi^{(\mu)}_{c}} \Kr^{n_{c-1}} \xrightarrow{\phi^{(\mu)}_{c-1}} \cdots
	\xrightarrow{\phi^{(\mu)}_{1}} \Kr^{n_0} \xrightarrow{\phi^{(\mu)}_0} M_{\mu} \xrightarrow{} 0
			\end{equation*}

	\end{itemize}
\end{theorem}
\begin{proof}
	Suppose that 
	$\phi_0$ is given by the following $p \times n_0$ matrix, 
	\begin{equation*}
		\phi_0 = 
		\begin{pmatrix}
			| & & | \\
			S_{1, 0} & \cdots & S_{n_0, 0} \\
			| & & |
		\end{pmatrix}
	\end{equation*}
	then $S_{1, 0}, \dots, S_{n_0, 0} \in \Kr^p$ form a minimal basis for $M$. 
	Similarly for $1 \leq i \leq c$, if $\phi_{i} : \Kr^{n_i} \rightarrow \Kr^{n_{i-1}}$ 
	is given by the $n_{i-1} \times n_i$ matrix, 
	\begin{equation*}
		\phi_0 = 
		\begin{pmatrix}
			| & & | \\
			S_{1, i} & \cdots & S_{n_i, i} \\
			| & & |
		\end{pmatrix},
	\end{equation*}
	then $S_{1, i}, \dots, S_{n_i, i} \in \Kr^{n_{i-1}}$ form a minimal basis for 
	the module \\ $Syz(S_{1, i-1}, \dots, S_{n_{i-1}, i-1})$. 
	Further, let $G_{1, 0}, \dots, G_{r_0, 0} \in \Kr^p$ be a standard basis for 
	$M$, and for 
	$1 \leq i \leq c-1$, let $G_{1, i} \dots, G_{r_i, i} \in \Kr^{n_{i-1}}$ be a standard 
	basis for $Syz(S_{1, i-1}, \dots, S_{n_{i-1}, i-1})$.
	Then, there exist $H_{ik, j} \in \Kr$ for $0 \leq j \leq c-1$, 
	$1 \leq k \leq n_j$, $1\leq i \leq r_j$ such that, 
	\begin{equation}
		\label{eq:main1}
		G_{i, j} = \sum_{k = 1}^{n_j} H_{ik, j} S_{k, j} \text{ for } 0 \leq j \leq c-1, 1 \leq i \leq r_j 
	\end{equation}
	Further, there exist $L_{ik, 0} \in \Kr$ for $1 \leq i \leq n_0$ and $1 \leq k \leq s$,  such that 
	\begin{equation}
		\label{eq:main2}
		S_{i, 0} = \sum_{k = 1}^{s} L_{ik, 0} F_k.
	\end{equation}
	For $0 \leq k \leq c-1$, Theorem \ref{thm:becker} implies that there exist monomials
	$P_{ij, k}(x), P_{ji, k}(x) \in \K[x]$ for $1 \leq i < j \leq n_k$ such that the 
	following is a standard representation, 
	\begin{equation}
		\label{eq:main3}
		P_{ij, k}(x) G_{i, k} - P_{ji, k}(x) G_{j, k} = \sum_{l = 1}^{r_k} Q_{ijl, k}G_{l, k} 
		\text{ for } 0 \leq k \leq c-1, 1 \leq i < j \leq n_k.
	\end{equation}
	Explicitly, the condition on the initial exponents for the above to be a standard basis is, 
	\begin{equation}
		\label{eq:sbcrit}
		\inexp(P_{ij, k}(x) G_{i, k} - P_{ji, k}(x) G_{j, k}) = \min\{ \inexp(Q_{ijl,k}G_{l, k}): 1 \leq l \leq r_k\} 
	\end{equation}
	Let $e_{1, k}, \dots, e_{n_k, k}$ be the 
	standard basis vectors of $\Kr^{n_k}$ for $0 \leq k \leq c-1$. 
	By Theorem \ref{thm:schreyer}, and \eqref{eq:main1}, the following are a basis of generators 
	of $Syz(S_{1, k}, \dots S_{n_k, k})$ 
	\begin{equation}
		\label{eq:main4}
		\begin{split}
			\Xi_{ij, k} =& P_{ij, k} (x) (\sum_{l = 1}^{n_k} H_{il, k} e_{l, k})  - 
			P_{ji, k}(x)(\sum_{l = 1}^{n_k} H_{jl, k} e_{l, k}) \\
		&- \sum_{l = 1}^{n_k} Q_{ijl, k}(\sum_{q = 1}^{n_k} H_{lq, k} e_{q, k})  \text{ for } 0 \leq k \leq c-1, 1 \leq i < j \leq r_k.
		\end{split}
	\end{equation}
	Since, $S_{1, k+1}, \dots, S_{n_{k+1}, k}$ are a minimal basis of generators of $Syz(S_{1, k}, \dots, S_{n_k, k})$, there exist $R_{ijl, k} \in \Kr$ such that the following relation holds, 
	\begin{equation}
		\label{eq:main5}
		S_{l, k+1} = \sum_{1 \leq i < j \leq r_k} \Xi_{ij, k}R_{ijl, k} \text{ for } 0 \leq k \leq c-1, 1\leq i \leq n_{k+1}
	\end{equation}
	Now treating the equations \eqref{eq:main1}, \eqref{eq:main2}, \eqref{eq:main3}, \eqref{eq:main4}, 
	and \eqref{eq:main5} as 
	a system of polynomial equations with polynomial coefficients in variables $S_{i, k}$, $G_{i,k}$, $H_{ij, k}$, 
	$Q_{ijl, k}$, and $L_{ij, 0}$, $\Xi_{ij,k}$ and $R_{ijl, k}$ with 
	coefficients in $\K[x]$ (for allowable values of the indices), the 
	above argument establishes the existence of a formal power series solution. 
	In what follows let $n_{-1} = p$ for simplicity. 
	Therefore, 
	for any integer $\mu \in \N$, by Artin's Approximation Theorem \cite[Theorem 1.10]{Ar2}
	there exist 
	\begin{itemize}
		\item[(1)] $F_i^{(\mu)} \in \Kra^p$ such that $j^{\mu}F_i^{(\mu)} = j^{\mu}F_i$ for 
			$1 \leq i \leq s$, 
		\item[(2)] $S^{(\mu)}_{i, k} \in \Kra^{n_{k-1}}$ such that $j^{\mu}S_{i,k}^{(\mu)} = j^{\mu}S_{i,k}i$ 
			for $0 \leq k \leq c$ and $1 \leq i \leq n_k$, 
		\item[(3)] $G^{(\mu)}_{i, k} \in \Kra^{n_{k-1}}$ such that $j^{\mu}G_{i,k}^{(\mu)} = j^{\mu}G_{i,k}i$ 
			for $0 \leq k \leq c-1$ and $1 \leq i \leq r_k$, 
		\item[(4)] $H^{(\mu)}_{ij, k} \in \Kra$ such that $j^{\mu}H_{ij,k}^{(\mu)} = j^{\mu}H_{ij,k}i$ 
			for $0 \leq k \leq c-1$, $1 \leq i \leq r_k$, and $1 \leq j \leq n_k$,
		\item[(5)] $Q^{(\mu)}_{ijl, k} \in \Kra$ such that $j^{\mu}Q_{ijl,k}^{(\mu)} = j^{\mu}Q_{ijl,k}i$ 
			for $0 \leq k \leq c-1$, $1 \leq l \leq r_k$ and $1 \leq i < j \leq r_k$,
		\item[(6)] $L_{ij, 0}^{(\mu)} \in \Kra$ such that $j^{\mu}L_{ij, 0}^{(\mu)} = j^{\mu}L_{ij, 0}$ for 
			$1 \leq i \leq n_0$ and $1 \leq j \leq s$, 
		\item[(7)] $\Xi_{ij, k}^{(\mu)} \in \Kra^{n_k}$ such that 
			$j^{\mu}\Xi_{ij, k}^{(\mu)} = j^{\mu}\Xi_{ij, k}$  for $0 \leq k \leq c-1$ and 
			$1 \leq i < j \leq r_k$, and 
		\item[(8)] $R_{ijl, k}^{(\mu)} \in \Kra$ such that $j^{\mu}R_{ijl, k}^{(\mu)} = j^{\mu} R_{ijl, k}$
			for $0 \leq k \leq c-1$, $1 \leq i < j \leq r_k$, and $1 \leq l \leq n_{k+1}$, 
	\end{itemize}
	that also solve the system of equations given by 
	\eqref{eq:main1}, \eqref{eq:main2}, \eqref{eq:main3}, \eqref{eq:main4} and \eqref{eq:main5}.  Let $(\tilde{\alpha}_k, i_k) = \max\{\inexp(G_{1, k}), \dots, \inexp(G_{r_k, k})\}$ for $0\leq k \leq c$,
	and let $\mu_0 = \max\{|\tilde{\alpha}_0|, \dots, |\tilde{\alpha}_c|\}$. Taking $\mu > \mu_0$, will 
	ensure that the approximants in item (3) above will satisfy the criterion \eqref{eq:sbcrit}, which 
	will, in turn, ensure that the corresponding approximants $\Xi^{(\mu)}_{ij, k}$ 
	for $1 \leq i < j \leq r_k$ in item (7) from a basis of 
	generators of $Syz(S^{(\mu)}_{1, k}, \dots, S^{(\mu)}_{n_k, k})$, for $0 \leq k \leq c-1$.
	Let $M_{\mu} = (F_1^{(\mu)}, \dots, F_s^{(\mu)})$.
	Now, for $0 \leq k \leq c$, define $\phi_k^{(\mu)} : \Kr^{n_k} \rightarrow \Kr^{n_{k-1}}$ to be the algebraic homomorphism 
	given by the matrix 
	\begin{equation*}
		\phi_k^{(\mu)} = 
		\begin{pmatrix}
			| & & |\\
			S_{1, k}^{(\mu)} & \dots & S_{n_k, k}^{(\mu)} \\
			| & & |
		\end{pmatrix}
	\end{equation*}
	Then, for $\mu \geq \mu_0$ one has the following
	\begin{itemize}
		\item[(1)] $G_{1, 0}^{(\mu)}, \dots, G_{r_0, 0}^{(\mu)}$ are a standard basis for 
			$M_{\mu}$. Also, the choice of $\mu_0$ implies that $\{\inexp(G_{1, 0}), \dots, 
				\inexp(G_{r_0, 0}) \} = \{\inexp(G_{1, 0}^{(\mu)}), \dots, 
				\inexp(G_{r_0, 0}^{(\mu)}) \}$, which implies that $\NN(M) = \NN(M_{\mu})$, 
				by point (iv) of Remark \ref{rem:sbb}. 
		\item[(2)] The condition for minimality of the resolution $\Fcal_{M}$, $\phi_i(\Kr^{n_i}) 
			\subseteq \mm \Kr^{n_{i-1}}$ for $1 \leq i \leq c$, implies that, 
			$\phi_k^{(\mu)}(\Kr^{n_{i}}) \subseteq \mm \Kr^{n_{i-1}}$. 
		\item[(3)] The injectivity of $\phi_c$ implies that the determinants of the 
			$n_{c} \times n_{c}$ minors of the matrix of $\phi_c$ are non-zero (Lemma 
			\ref{lem:injective}), which implies that the determinants of the 
			$n_{c} \times n_{c}$ minors of the matrix of $\phi_c^{(\mu)}$ are 
			non-zero, (because the determinant is just a polynomial function of the 
			entries of a matrix), which implies that $\phi_c^{(\mu)}$ is injective.  
		\item[(4)] Items (2), and (3) above imply that the following is a minimal free resolution 
			of the module $M_{\mu}$:
			\begin{equation*}
		\Fcal_{M_{\mu}}: 0 \xrightarrow{} \Kr^{n_c} \xrightarrow{\phi^{(\mu)}_{c}} \Kr^{n_{c-1}} \xrightarrow{\phi^{(\mu)}_{c-1}} \cdots
	\xrightarrow{\phi^{(\mu)}_{1}} \Kr^{n_0} \xrightarrow{\phi^{(\mu)}_0} M_{\mu} \xrightarrow{} 0.
			\end{equation*}
	\end{itemize}
\end{proof}

\section{Regarding convergence}
\label{sec:conv}
In the case when $\K$ is a complete real valued field, it is possible to define the 
ring of convergent power series $\K\{x\}$. Theorem \ref{thm:main} remains valid 
if $\Kr$ is replaced by $\K\{x\}$. Specifically, 
\begin{remark}
	\begin{itemize}
		\item[(i)] Theorem \ref{thm:becker} depends on Hironaka's Division Theorem, 
			which is valid for $\K\{x\}$ by \cite[Theorem 3.4]{BM1}.  
		\item[(ii)] Theorem \ref{thm:schreyer} is also only dependent on \cite[Theorem 3.4]{BM1}, 
			and is proved in this context as \cite[Theorem 6.2]{BM1}. 
		\item[(iii)] All the theory in Section \ref{sec:freeres} is valid for $\Kr$ replaced 
			by any regular local ring, in particular, $\K\{x\}$.  
		\item[(iv)] Artin's Approximation Theorem \cite[Theorem 1.10]{Ar2} used in 
			the proof of Theorem \ref{thm:main} is also valid for approximations of 
			solutions in $\K\{x\}$ of polynomial systems of equations. 
	\end{itemize}
\end{remark}

\section{Applications}
\label{sec:appl}

\subsection{Approximation of Cohen-Macaulay algebras}
\label{sec:cmapprox}

\begin{theorem}
	\label{thm:cmbettiapprox}
	Let $I = (F_1, \dots, F_s)$ be an ideal in $\Kr$ such that $\Kr/I$ is Cohen-Macaulay, has 
	dimension $\dim \Kr/I = n-k$ and has minimal Betti numbers $\beta_0^{\Kr/I}, \dots, \beta_k^{\Kr/I}$. 
	Then, there exists an integer $\mu_0$ such that for each $\mu \geq \mu_0$ there exist, 
	$F^{(\mu)}_1, \dots, F_s^{(\mu)} \in \Kra$ which generate an ideal $I_{\mu} \subseteq \Kr$ such that, 
	\begin{itemize}
		\item[(i)] $j^{\mu}F^{(\mu)}_i = j^{\mu} F_i$ for $1 \leq i \leq s$. 
		\item[(ii)] The ring $\Kr/I_{\mu}$ is Cohen-Macaulay with $\dim \Kr/I_{\mu} = \dim \Kr/I$.
		\item[(iii)] The minimal Betti numbers of $\Kr/I_{\mu}$ satisfy 
			$\beta_i^{\Kr/I_{\mu}} = \beta_i^{\Kr/I}$ for $0 \leq i \leq k$. 
		\item[(iv)] $H_{I_{\mu}} = H_I$.
	\end{itemize}
\end{theorem}
\begin{proof}
	Let the following be a minimal free resolution of $I$
	\begin{equation*}
		\Fcal_{I}: 0 \xrightarrow{} \Kr^{n_c} \xrightarrow{\phi_{c}} \Kr^{n_{c-1}} \xrightarrow{\phi_{c-1}} \cdots
		\xrightarrow{\phi_{1}} \Kr^{n_0} \xrightarrow{\phi_0} I \xrightarrow{} 0.
	\end{equation*}
	Then by Remark \ref{rem:idealres}, and Theorem \ref{thm:cmthm}, $c=k-1$ and the following is a minimal free resolution of $\Kr/I$ 
	\begin{equation*}
		\Fcal_{\Kr/I}: 0 \xrightarrow{} \Kr^{n_{k-1}} \xrightarrow{\phi_{k-1}} \cdots
		\xrightarrow{\phi_{1}} \Kr^{n_0} \xrightarrow{\phi_0} \Kr \xrightarrow{\pi} \Kr/I \rightarrow 0.
	\end{equation*}
	Now, by Theorem \ref{thm:main}, there exists $\mu_0 \in \N$, such that for each $\mu \geq \mu_0$ there exist, 
	$F_1^{(\mu)}, \dots, F_s^{(\mu)} \in \Kra$, generating 
	$I_{\mu} \subseteq \Kr$ such that $j^{\mu}F_i^{(\mu)} = j^{\mu}F_i$, $\NN(I) = \NN(I_{\mu})$, and the following is a minimal 
	free resolution of $I_{\mu}$ 
	\begin{equation*}
		\Fcal_{I_{\mu}}: 0 \xrightarrow{} \Kr^{n_{k-1}} \xrightarrow{\phi^{(\mu)}_{k-1}} \Kr^{n_{k-2}} \xrightarrow{\phi^{(\mu)}_{k-2}} \cdots
	\xrightarrow{\phi^{(\mu)}_{1}} \Kr^{n_0} \xrightarrow{\phi^{(\mu)}_0} I_{\mu} \xrightarrow{} 0,
	\end{equation*}
	where $\phi_{i}^{(\mu)}$ are algebraic homomorphisms for $0 \leq i \leq k-1$.
	This implies that 
	\begin{itemize}
		\item[(1)] The following is a minimal free resolution of $\Kr/I_{\mu}$ 
			\begin{equation*}
				\Fcal_{\Kr/I_{\mu}}: 0 \xrightarrow{} \Kr^{n_{k-1}} \xrightarrow{\phi^{(\mu)}_{k-1}} \cdots
				\xrightarrow{\phi^{(\mu)}_{1}} \Kr^{n_0} \xrightarrow{\phi^{(\mu)}_0} \Kr \xrightarrow{\pi} \Kr/I_{\mu} \rightarrow 0.  
			\end{equation*}
		\item[(2)] $H_{I_{\mu}} = H_{I}$, by Lemma \ref{lem:hsfunc}. 	
	\end{itemize}
	Item (2) above implies that the Hilbert-Samuel polynomials of $I$ and $I_{\mu}$ and the same 
	which by Remark \ref{rem:HSpoly} implies that $\dim \Kr/I = \dim \Kr/I_{\mu} = n-k$. Now, because 
	of item (1) above, Theorem \ref{thm:cmthm} 
	implies that $\Kr/I_{\mu}$ is Cohen-Macaulay, 
	Item (1) also implies that $\beta_i^{\Kr/I_{\mu}} = \beta_i^{\Kr/I}$, for all $0 \leq i \leq k$.  
\end{proof}

The above implies in particular that the Cohen-Macaulay 
type of $\Kr/I$ and $\Kr/I_{\mu}$ are the same for all $\mu \geq \mu_0$. 
By \cite[Corollary 21.16]{Eis}, $\Kr/I$ is Gorenstein if and only if the Cohen-Macaulay type 
of $\Kr/I$ is one. Therefore, an immediate corollary of the above is 
\begin{corollary}
	\label{cor:gorensteinapprox}
	Let $I = (F_1, \dots, F_s)$ be an ideal in $\Kr$ such that $\Kr/I$ is Gorenstein, and has 
	dimension $\dim \Kr/I = n-k$. 
	Then, there exists an integer $\mu_0$ such that for each $\mu \geq \mu_0$ there exist, 
	$F^{(\mu)}_1, \dots, F^{(\mu)} \in \Kra$ which generate an ideal $I_{\mu} \in \Kr$ such that, 
	\begin{itemize}
		\item[(i)] $j^{\mu}F^{(\mu)}_i = j^{\mu} F_i$ for $1 \leq i \leq s$. 
		\item[(ii)] The ring $\Kr/I_{\mu}$ is Gorenstein with $\dim \Kr/I_{\mu} = \dim \Kr/I$.
		\item[(iv)] $H_{I_{\mu}} = H_I$.
	\end{itemize}
\end{corollary}

\begin{remark}
	\begin{itemize}
		\item[(i)] In the case where $\K$ is a complete real-valued field, by the remarks in Section \ref{sec:conv},
			Theorem \ref{thm:cmbettiapprox} and Theorem \ref{cor:gorensteinapprox} remain true when $\Kr$ is replaced by $\K\{x\}$.
		\item[(ii)] The result \cite[Theorem 8.1]{JP} follows immediately from Theorem \ref{thm:cmbettiapprox}.
	\end{itemize}
\end{remark}

\subsection{Approximation of flat maps}
\label{sec:flatapprox}
Throughout this section $m$ will be a fixed integer and 
$y$ will denote the $m$-tuple of variables $(y_1, \dots, y_m)$. 
Consider a homomorphism rings $\phi:\K[[y]] \rightarrow \Kr/I$, 
for some ideal $I \subseteq \Kr$. 
Such a map can be completely specified by 
specifying the images of $y_1, \dots, y_m$.
The images of $y_1, \dots, y_m$ under $\phi$ can 
belong to $\Kr/I$. Let  
$\pi: \Kr \rightarrow \Kr/I$ be the canonical 
projection. A homomorphism 
$\phi: \K[[y]] \rightarrow \Kr/I$ can be 
completely specified by giving 
$\tilde{\phi}(y_i) = \phi_i (x) \in \Kr$, for $1\leq i \leq m$, 
for some $\tilde{\phi}: \K[[y]]\rightarrow \Kr$ such 
that $\phi = \pi \circ \tilde{\phi}$. In such a case 
the homomorphism $\phi$ is said to be \emph{defined by} 
$\phi_i (x) \in \Kr$ for $1\leq i \leq m$.  
The homomorphism $\phi$ is called 
\emph{algebraic} if it is defined by 
$\phi_i (x) \in \Kra$ for $1\leq i \leq m$.  
Note that the special fibre of a homomorphism 
$\phi: \K[[y]] \rightarrow \Kr/I$ is 
isomorphic to $\Kr/(I + J)$ where, 
$J$ is the ideal generated by 
any set of power series $\phi_i (x) \in \Kr$
defining $\phi$. 

\begin{theorem}
	\label{thm:flatapprox}
	Suppose that $I = (F_1, \dots, F_s) \subseteq \Kr$ is an ideal, $\Kr/I$ is Cohen-Macaulay and that $\phi: \K[[y]] \rightarrow \Kr/I$ is 
	flat homomorphism of rings, defined by $\phi_i(x) \in \Kr$ for $1 \leq i \leq m$. Then there exists a $\mu_0 \in \N$ such 
	that for each $\mu \geq \mu_0$ there exist $F^{(\mu)}_1, \dots, F^{(\mu)}_s \in \Kra$ which generate an ideal $I_{\mu} \subseteq \Kr$ 
	and an algebraic homomorphism $\phi_{\mu} : \K[[y]] \rightarrow \Kr$ defined by $\phi_{i}^{(\mu)} \in \Kra$, $1\leq i \leq m$, such that, 
	\begin{itemize}
		\item[(i)] $j^{\mu} F_i^{(\mu)} = j^{\mu} F_i$ for $1\leq i \leq s$. 
		\item[(ii)] $j^{\mu} \phi_{i}^{(\mu)} = j^{\mu} \phi_i$ for $1 \leq i \leq m$. 
		\item[(iii)] $\phi_{\mu}$ is a flat homomorphism.  
		\item[(iv)] If $J = (F_1, \dots, F_s, \phi_1, \dots, \phi_m)$ and $J_{\mu} = (F_1^{(\mu)}, \dots, F_s^{(\mu)}, \phi_1^{(\mu)}, \dots, \phi_m^{(\mu)})$, 
			then $H_J = H_{J_{\mu}}$ and the minimal Betti numbers of $\Kr/J$ and $\Kr/J_{\mu}$ are 
			the same. 
		\item[(v)] $H_I = H_{I_{\mu}}$ and the minimal Betti numbers of $\Kr/I$ and $\Kr/I_{\mu}$ are
			the same. 
	\end{itemize}
\end{theorem}
\begin{proof}
	By the flatness criterion, \cite[Theorem B.8.11]{GLS}, the flatness of $\phi$, and the Cohen-Macaulayness of $\Kr/I$ imply that, 
	\begin{equation*}
		\dim \Kr/I = \dim \K[[y]] + \dim \Kr/J, 
	\end{equation*}
	where $J$ is defined as in point (iv) in the statement of the theorem.  
	Further, by \cite[Theorem B.8.15]{GLS}, the flatness of $\phi$, and the 
	Cohen-Macaulayness of $\Kr/I$ imply that $\Kr/J$ is also Cohen-Macaulay.
	Now, in the proofs of Theorem \ref{thm:main} and Theorem \ref{thm:cmbettiapprox}, the systems of equations 
	corresponding to the approximations of $I = (F_1, \dots, F_s)$ and 
	$J = (F_1, \dots, F_s, \phi_1, \dots, \phi_m)$ can be 
	solved simultaneously, and thus, there exists $\mu_0 \in \N$, 
	such that for each $\mu \geq \mu_0$ there exist
	$F^{(\mu)}_1, \dots, F^{(\mu)}_s, \phi^{(\mu)}_1, \dots, \phi^{(\mu)}_m \in \Kra$ such 
	that if $I_{\mu} = (F^{(\mu)}_1, \dots, F^{(\mu)}_s)$ and $J_{\mu} = (F^{(\mu)}_1, \dots, F^{(\mu)}_s, \phi^{(\mu)}_1, \dots, \phi^{(\mu)}_m)$, 
	then $H_I = H_{I_{\mu}}$ and $H_{J} = H_{J_{\mu}}$, $\dim \Kr/I = \dim \Kr/I_{\mu}$, $\dim \Kr/J = \dim \Kr/J_{\mu}$, and
	$\Kr/I_{\mu}$ and $\Kr/J_{\mu}$ are both Cohen-Macaulay. 
	Also, the minimal Betti numbers of $\Kr/I$ and $\Kr/J$ are the same as those of $\Kr/I_{\mu}$ and 
	$\Kr/J_{\mu}$ respectively. 
	By \cite[Theorem B.8.11]{GLS} again, the Cohen-Macaulayness of $\Kr/I_{\mu}$ and the 
	relationship between the dimensions above, imply that the homomorphism  
	$\phi_{\mu}: \K[[y]] \rightarrow \Kr/I_{\mu}$ defined by $\phi_{i}^{(\mu)}$ for $1\leq i \leq m$ is flat. 
\end{proof}
\begin{remark}
	\begin{itemize}
		\item[(i)] When $\K$ is a complete 
			real valued field, by the comments in Section \ref{sec:conv}, Theorem \ref{thm:flatapprox} remains 
			valid if $\K[[y]]$ and $\Kr$ are replaced by $\K\{y\}$ and $\K\{x\}$ respectively. 
		\item[(ii)] The result \cite[Theorem 1.2]{AYP} follows from Theorem \ref{thm:flatapprox} above immediately. 
	\end{itemize}
\end{remark}

\bibliographystyle{amsplain}

\bibliography{res_approx}{}

%

\end{document}